\documentclass[a4paper]{amsart}

\usepackage{amsfonts}
\usepackage{amssymb}
\usepackage{amsthm}
\usepackage{amsmath}
\usepackage[all]{xy}

\SelectTips{cm}{}

\theoremstyle{plain}
\newtheorem{theorem}{Theorem}[section]
\newtheorem{proposition}[theorem]{Proposition}

\newtheorem{lemma}[theorem]{Lemma}
\newtheorem{conjecture}[theorem]{Conjecture}

\newtheorem{example}[theorem]{Example}

\theoremstyle{definition}
\newtheorem{definition}[theorem]{Definition}

\newcommand{\cc}{\mathbb{C}}
\newcommand{\nn}{\mathbb{N}}
\newcommand{\rr}{\mathbb{R}}
\newcommand{\zz}{\mathbb{Z}}
\newcommand{\even}{\operatorname{even}}
\newcommand{\odd}{\operatorname{odd}}
\newcommand{\colim}{\operatorname{colim}}
\newcommand{\im}{\operatorname{im}}
\newcommand{\tr}{\operatorname{tr}}
\newcommand{\id}{\operatorname{id}}

\begin{document}

\title{$L^2$-Invariants of Finite Aspherical CW-Complexes}

\author{Christian Wegner}



\address{Mathematisches Institut \\ Universit\"at M\"unster \\
Einsteinstra{\ss}e 62 \\ M\"unster, D-48149 \\ Germany}
\email{c.wegner@uni-muenster.de}

\maketitle

\begin{abstract}
Let $X$ be a finite aspherical CW-complex whose fundamental group $\pi_1(X)$ possesses a subnormal series $\pi_1(X) \rhd G_m \rhd \ldots \rhd G_0$ with a non-trivial elementary amenable group $G_0$. We investigate the $L^2$-invariants of the universal covering of such a CW-complex $X$. We show that the Novikov-Shubin invariants $\alpha_n({\tilde X})$ are positive. We further prove that the $L^2$-torsion $\rho^{(2)}({\tilde X})$ vanishes if $\pi_1(X)$ has semi-integral determinant.
\end{abstract}

\section{Introduction}
\label{intro}

The present paper investigates $L^2$-invariants, like $L^2$-Betti numbers, Novikov-Shubin invariants and $L^2$-torsion, of the universal covering of finite aspherical CW-complexes. The main result of this paper is related to the following conjecture posed by L\"uck (see \cite[Conjecture 9.21]{Luc02}). In case of $L^2$-Betti numbers this conjecture goes back to a question of Gromov (see \cite[Section 8A]{Gro93}).

\begin{conjecture}
Let $M$ be a closed aspherical orientable manifold with vanishing simplicial volume. Then the $L^2$-Betti numbers $b^{(2)}_n({\tilde M})$ vanish, the Novikov-Shubin invariants $\alpha_n({\tilde M})$ are positive and the $L^2$-torsion $\rho^{(2)}({\tilde M})$ vanishes.
\end{conjecture}

Since any closed orientable manifold with non-trivial amenable fundamental group has vanishing simplicial volume (see \cite[Corollary (C) on page 40]{Gro82}), it is obvious to consider finite aspherical CW-complexes with amenable or elementary amenable fundamental group. Recall that the class of elementary amenable groups is the smallest class of groups containing all finite groups and abelian groups, being closed under extensions and under upwards directed unions. Any elementary amenable group is amenable.

In this paper we prove the following theorem.

\begin{theorem} \label{maintheo}
Let $X$ be a finite aspherical CW-complex. Suppose that the fundamental group $\pi_1(X)$ possesses a subnormal series
\[
\pi_1(X) \rhd G_m \rhd \ldots \rhd G_0
\]
with a non-trivial elementary amenable group $G_0$. Then the following statements hold.
\begin{enumerate}
\item $b_n^{(2)}({\tilde X})=0$ for all $n \geq 0$. \label{maintheo1}
\item $\alpha_n({\tilde X})>0$ for all $n \geq 1$. \label{maintheo2}
\item If $\pi_1(X)$ has semi-integral determinant then $\rho^{(2)}({\tilde X})=0$. \label{maintheo3}
\end{enumerate}
\end{theorem}

For a large class of groups it is known that the members have semi-integral determinant (see Proposition \ref{classG}). It is conjectured that all groups have semi-integral determinant.

This article is organized as follows. 
In section \ref{Basics} we give a short introduction to $L^2$-invariants. 
In section \ref{NovShuInv} we will consider Novikov-Shubin invariants. The main result of that section is Theorem \ref{theoNS}. The second part of this theorem coincides with assertion \ref{maintheo2} of Theorem \ref{maintheo}. The proof is a combination of a group theoretical result of Hillman and Linnell (see Proposition \ref{HL}) and results on Novikov-Shubin invariants of L\"uck, Reich and Schick published in the paper \cite{LRS99}.
In section \ref{L2Torsion} we will prove the other two assertions of Theorem \ref{maintheo} using localization techniques for non-commutative rings. The main results of this section are Theorem \ref{rhovan} and Theorem \ref{rhomain}.

Theorem \ref{maintheo} is a slight generalization of the main result of the author's dissertation at the university of M\"unster.

\section{Some Basics on $L^2$-Invariants}
\label{Basics}

In this section we give a short introduction to $L^2$-invariants. We define $L^2$-Betti numbers, Novikov-Shubin invariants and $L^2$-torsion of finite CW-complexes. For a more detailed survey on $L^2$-invariants we refer to the article \cite{Luc02}.

For a discrete group $G$ we define $l^2(G)$ as the Hilbert space completion of the complex group ring ${\cc}G$ with respect to the inner product
\[
\langle \sum_{g \in G} c_g g , \sum_{g \in G} d_g g \rangle := \sum_{g \in G} \overline{c_g} d_g.
\]
The group von Neumann algebra ${\mathcal N}G$ is the space of $G$-equivariant bounded operators from $l^2(G)$ to itself.

\begin{definition}
A \emph{finitely generated Hilbert ${\mathcal N}G$-module} $V$ is a Hilbert space $V$ together with a $G$-action such that there exists an isometric $G$-equivariant embedding of $V$ into $l^2(G)^n$ for some $n \in \nn$.
A \emph{morphism of Hilbert ${\mathcal N}G$-modules} is a bounded $G$-equivariant operator. A \emph{weak isomorphism} is a morphism which is injective and has dense image.
\end{definition}

\begin{definition}
Let $f: V \to V$ be a positive endomorphism of a finitely generated Hilbert ${\mathcal N}G$-module.
Choose any orthogonal $G$-equivariant projection $pr: l^2(G)^n \to l^2(G)^n$ whose image is isometrically $G$-isomorphic to $V$. Let $\overline{f}: l^2(G)^n \to l^2(G)^n$ be the positive operator given by the composition
\[
\overline{f}: l^2(G)^n {\mathop{\longrightarrow}\limits^{pr}} \im(pr) \cong V {\mathop{\longrightarrow}\limits^{f}} V \cong \im(pr) \hookrightarrow l^2(G)^n.
\]
Define the \emph{von Neumann trace} of $f$ by
\[
\tr_{{\mathcal N}G}(f) := \sum_{k=1}^n \langle \overline{f}(e_k) , e_k \rangle \in [0,\infty)
\]
where $e_k \in l^2(G)^n$ denotes the element whose components are zero except for the $k$-th entry, which is the unit in ${\cc}G \subseteq l^2(G)$.
\end{definition}

Using the von Neumann trace we can define the von Neumann dimension for Hilbert ${\mathcal N}G$-modules.

\begin{definition}
The \emph{von Neumann dimension} of a finitely generated Hilbert ${\mathcal N}G$-module $V$ is given by
\[
\dim_{{\mathcal N}G}(V) := \tr_{{\mathcal N}G}(id: V \to V) \in [0,\infty).
\]
\end{definition}
The von Neumann dimension satisfies faithfulness, monotony, continuity and weak exactness (see \cite[Lemma 1.4]{Luc02}).

Let $X$ be a finite CW-complex with fundamental group $G$. Then we get a Hilbert ${\mathcal N}G$-chain complex $C^{(2)}_*({\tilde X}) := l^2(G) \otimes_{\zz G} C_*({\tilde X})$, where $C_*({\tilde X})$ is the cellular chain complex of the universal covering of $X$. Notice that $C_n({\tilde X})$ is a finite free $\zz G$-module with basis given by a cellular structure of $X$.
\begin{definition}
We define the $n$-th (reduced) $L^2$-homology and the $n$-th $L^2$-Betti number of a finite CW-complex $X$ by
\begin{eqnarray*}
H_n^{(2)}({\tilde X}) &:=& \ker(c_n^{(2)}({\tilde X})) / \overline{\im(c_{n+1}^{(2)}({\tilde X}))},\\
b_n^{(2)}({\tilde X}) &:=& \dim_{{\mathcal N}G}(H_n^{(2)}({\tilde X})).
\end{eqnarray*}
\end{definition}

In the definition of $H_n^{(2)}({\tilde X})$ we divide by the closure of the image to get the structure of a Hilbert ${\mathcal N}G$-module on $H_n^{(2)}({\tilde X})$.
Notice that $H_n^{(2)}({\tilde X})$ is isometrically $G$-isomorphic to the kernel of the combinatorial Laplace operator, i.e.
\[
H_n^{(2)}({\tilde X}) = \ker(c^{(2)}_n({\tilde X})^* \circ c^{(2)}_n({\tilde X}) + c^{(2)}_{n+1}({\tilde X}) \circ c^{(2)}_{n+1}({\tilde X})^*).
\]
The basic properties of the $L^2$-Betti numbers like homotopy invariance, Euler-Poincar{´e} formula or multiplicativity under finite coverings are described in \cite[Theorem 1.7]{Luc02}.

For the definition of the Novikov-Shubin invariants and the $L^2$-torsion we need spectral density functions.
\begin{definition}
Let $f: U \to V$ be a morphism of finitely generated Hilbert ${\mathcal N}G$-modules. Denote by $\{E_\lambda^{f^*f}: U \to U \mid \lambda \in \rr\}$ the family of spectral projections of the positive endomorphism $f^*f$. Define the \emph{spectral density function} of $f$ by
\[
F(f): \rr \to [0,\infty), \lambda \mapsto \dim_{{\mathcal N}G}(\im(E_{\lambda^2}^{f^*f})).
\]
\end{definition}
The spectral density function is monotonous and right-continuous. It defines a measure on the Borel $\sigma$-algebra on $\rr$ which is uniquely determined by $dF(f)((a,b]) := F(f)(b)-F(f)(a)$ for $a<b$.

\begin{definition}
Let $X$ be a finite CW-complex with fundamental group $G$. We define its \emph{Novikov-Shubin invariants} by
\[
\alpha_n({\tilde X}) := \liminf_{\lambda \to 0^+} \frac{\ln(F(c_n^{(2)}({\tilde X}),\lambda)-F(c_n^{(2)}({\tilde X}),0))}{\ln(\lambda)} \in [0,\infty],
\]
if $F(c_n^{(2)}({\tilde X}),\lambda) > F(c_n^{(2)}({\tilde X}),0)$ holds for all $\lambda > 0$. Otherwise we set $\alpha_n({\tilde X}) := \infty^+$ where
$\infty^+$ is a new formal symbol.
\end{definition}

For the basic properties of the Novikov-Shubin invariants like homotopy invariance or invariance under finite coverings we refer to \cite[Theorem 8.8]{Luc02}.

\begin{definition}
Let $f: U \to V$ be a morphism of finitely generated Hilbert ${\mathcal N}G$-modules. We define the \emph{determinant} of $f$ by
\[
{\textstyle\det}_{{\mathcal N}G}(f) := \exp(\int_{0^+}^\infty \ln(\lambda) \, dF(f)(\lambda))
\]
if $\int_{0^+}^\infty \ln(\lambda) \, dF(f)(\lambda) > -\infty$ and by $\det_{{\mathcal N}G}(f) := 0$ otherwise.
\end{definition}

We use this determinant to define the $L^2$-torsion. Some properties of this determinant can be found in \cite[Lemma 9.7]{Luc02}.

\begin{definition}
Let $X$ be a finite CW-complex with fundamental group $G$. Suppose that $b_n^{(2)}({\tilde X}) = 0$ and $\det_{{\mathcal N}G}(c_n^{(2)}({\tilde X})) > 0$ for all $n$. We define its \emph{$L^2$-torsion} by 
\[
\rho^{(2)}({\tilde X}) := - \sum_{n \geq 0} (-1)^n \cdot \ln({\textstyle\det}_{{\mathcal N}G}(c_n^{(2)}({\tilde X}))) \in \rr.
\]
\end{definition}

The $L^2$-torsion can also expressed in terms of the Laplace operators $\Delta_n := c^{(2)}_n({\tilde X})^* \circ c^{(2)}_n({\tilde X}) + c^{(2)}_{n+1}({\tilde X}) \circ c^{(2)}_{n+1}({\tilde X})^*$:
\[
\rho^{(2)}({\tilde X}) := - \frac{1}{2} \cdot \sum_{n \geq 0} (-1)^n \cdot n \cdot \ln({\textstyle\det}_{{\mathcal N}G}(\Delta_n)).
\]
A priori it is not clear if the $L^2$-torsion is a homotopy invariance. But it is known to be true if the fundamental group lies in a large class of group as we will see in the following propositions.

\begin{definition}
A group $G$ has \emph{semi-integral determinant} if for all matrices $A \in M(m \times n;{\zz}G)$ the determinant of the morphism $r_A^{(2)}: l^2(G)^m \to l^2(G)^n$ given by right multiplication with $A$ satisfies $\det_{{\mathcal N}G}(r_A^{(2)}) \geq 1$.
\end{definition}

It is conjectured that all groups have semi-integral determinant. The following propositions are proved in \cite{Sch98}.

\begin{proposition}\label{classG}
Let ${\mathcal G}$ be the smallest class of groups containing the trivial group and being closed under the following operations:
\begin{enumerate}
\item Amenable extensions\\ Let $H \subset G$ be a subgroup. Suppose that $H \in {\mathcal G}$ and that the quotient $G/H$ is an amenable homogeneous space. Then $G \in {\mathcal G}$.
\item Colimits\\ If $G = \colim_{i \in I} G_i$ and $G_i \in {\mathcal G}$ for all $i \in I$ then $G \in {\mathcal G}$.
\item Inverse limits\\ If $G = \lim_{i \in I} G_i$ and $G_i \in {\mathcal G}$ for all $i \in I$ then $G \in {\mathcal G}$.
\end{enumerate}
Any group $G \in {\mathcal G}$ has semi-integral determinant.
\end{proposition}

Notice that the class of groups ${\mathcal G}$ contains all amenable groups, is residually closed and is closed under taking subgroups, forming direct and inverse limits of directed systems and forming direct and free products.

\begin{proposition}
\begin{enumerate}
\item Let $X$ be a finite CW-complex such that $\pi_1(X)$ has semi-integral determinant. Then $\det_{{\mathcal N}G}(c_n^{(2)}({\tilde X})) > 0$ for all $n$.
\item Let $f: X \to Y$ be a homotopy equivalence of finite CW-complexes such that $\pi_1(X) \cong \pi_1(Y)$ has semi-integral determinant. Suppose that $b_n^{(2)}({\tilde X}) = b_n^{(2)}({\tilde Y}) = 0$ for all $n$. Then $\rho^{(2)}({\tilde X}) = \rho^{(2)}({\tilde Y})$.
\end{enumerate}
\end{proposition}

\section{Novikov-Shubin Invariants}
\label{NovShuInv}

The main result of this section is Theorem \ref{theoNS}. The proof is a combination of a remarkable group theoretical result of Hillman and Linnell (Proposition \ref{HL}) and results of 
L\"uck, Reich and Schick published in the paper \cite{LRS99}. In \cite[Definition 2.2]{LRS99} a capacity $c(M)$ for ${\mathcal N}G$-modules $M$ is introduced, where ${\mathcal N}G$ denotes the group von Neumann algebra of the group $G$. These invariants take values in $\{ 0^- \} \coprod [0,\infty]$ where $0^-$ is a new formal symbol. The relation between the capacity function and the Novikov-Shubin invariants is as follows (compare \cite[Theorem 6.1]{Luc97}):
Let $X$ be a finite CW-complex. Denote by $H_n^{\pi_1(X)}({\tilde X};{\mathcal N}\pi_1(X))$ the ${\mathcal N}\pi_1(X)$-module given by the $n$-th homology of the chain complex ${\mathcal N}\pi_1(X) \otimes_{\zz \pi_1(X)} C_*^{sing}({\tilde X})$ where $C_*^{sing}({\tilde X})$ is the singular chain complex of ${\tilde X}$. Then
\[
\alpha_{n}({\tilde X}) = c(H_{n-1}^{\pi_1(X)}({\tilde X};{\mathcal N}\pi_1(X)))^{-1} \mbox{ for all } n \geq 1.
\]
(Here $0^{-1}$ and $(0^-)^{-1}$ are defined as $0^{-1} := \infty$ and $(0^-)^{-1} := \infty^+$.)\\
The property above allows to extend the Novikov-Shubin invariants $\alpha_n(Y)$ to arbitrary topological spaces $Y$ with an action of a group $G$ by the expression $c(H_{n-1}^G(Y;{\mathcal N}G))^{-1}$.

We now consider the capacities of $H_n^G(EG;{\mathcal N}G)$ where $EG$ is any universal free $G$-space. In view of the next proposition we notice that there is a technical notion for ${\mathcal N}G$-modules ``cofinal-measurable'' defined in \cite[Definition 2.1]{LRS99}.
\begin{proposition} \label{LRS}
\begin{enumerate}
\item Let $k \geq 1$. The ${\mathcal N}\zz^k$-module $H_n^{\zz^k}(E\zz^k;{\mathcal N}\zz^k)$ is cofinal-measurable for all $n \geq 0$ and
    \[
    c(H_n^{\zz^k}(E\zz^k;{\mathcal N}\zz^k)) = \left\{ \begin{array}{ll} 1/k & \mbox{ if } 0 \leq n \leq k-1 \\ 0^- & \mbox{ if } n \geq k \end{array} \right. \label{LRS1}
    \]
\item If there is a cofinal system of subgroups $H < G$ with cofinal-measurable $H_n^H(EH;{\mathcal N}H)$, then $H_n^G(EG;{\mathcal N}G)$ is cofinal-measurable and
    \[
    c(H_n^G(EG;{\mathcal N}G)) \leq \liminf \{ c(H_n^H(EH;{\mathcal N}H)) \}. \label{LRS2}
    \]
\item Let $H \lhd G$ be a normal subgroup. Suppose that $H_n^H(EH;{\mathcal N}H)$ is cofinal-measurable for all $n \geq 0$. Then $H_n^G(EG;{\mathcal N}G)$ is cofinal-measurable for all $n \geq 0$ and
    \[
    c(H_n^G(EG;{\mathcal N}G)) \leq \sum_{i=0}^n c(H_i^H(EH;{\mathcal N}H)). \label{LRS3}
    \]
\end{enumerate}
\end{proposition}
\begin{proof}
A proof is given in \cite[Theorem 3.7]{LRS99}.
\end{proof}

For the proof of Theorem \ref{theoNS} we need the following group theoretical result which is due to Hillman and Linnell.
\begin{proposition} \label{HL}
A group of finite cohomological dimension which has a non-trivial elementary amenable normal subgroup has a non-trivial torsionfree abelian normal subgroup.
\end{proposition}
\begin{proof}
For a proof see \cite[Corollary 2]{HL92}.
\end{proof}

\begin{theorem} \label{theoNS}
Let $X$ be a finite aspherical CW-complex.
\begin{enumerate}
\item If the fundamental group $\pi_1(X)$ possesses a non-trivial elementary amenable normal subgroup, then $\alpha_n({\tilde X}) \geq 1$ for all $n \geq 1$.
\item If the fundamental group $\pi_1(X)$ possesses a subnormal series
	\[
	\pi_1(X) \rhd G_m \rhd \ldots \rhd G_0
	\]
	with a non-trivial elementary amenable group $G_0$, then $\alpha_n({\tilde X})>0$ for all $n \geq 1$.
\end{enumerate}
\end{theorem}
\begin{proof}
Let us first consider a non-trivial torsionfree abelian group $H$. Since any non-trivial finitely generated subgroup of $H$ is isomorphic to $\zz^k$ for some $k \geq 1$, we conclude from Proposition \ref{LRS} (\ref{LRS1}) and (\ref{LRS2}) that $H_n^H(EH;{\mathcal N}H)$ is cofinal-measurable for all $n \geq 0$ and satisfies
\begin{equation}
\sum_{i=0}^n c(H_i^H(EH;{\mathcal N}H)) \leq 1 \mbox{ for all } n \geq 0. \label{eq_H}
\end{equation}
Now the statements are a direct consequence of the equation
\[
\alpha_{n}({\tilde X}) = c(H_{n-1}^{\pi_1(X)}(E\pi_1(X);{\mathcal N}\pi_1(X)))^{-1} \mbox{ for all } n \geq 1,
\]
Proposition \ref{HL}, Proposition \ref{LRS} (\ref{LRS3}) and equation \ref{eq_H}.
\end{proof}

\section{$L^2$-Betti Numbers and $L^2$-Torsion}
\label{L2Torsion}

One of the main results of this section is Theorem \ref{rhovan}. It shows the vanishing of $L^2$-Betti numbers and $L^2$-torsion if a certain localization property is satisfied.
This theorem allows us to prove Theorem \ref{rhomain} which coincides with the assertions \ref{maintheo1} and \ref{maintheo3} of Theorem \ref{maintheo}.

\begin{definition}
A subset $S \subset \zz G$ has the \emph{determinant localization property} if the following conditions are fulfilled:
\begin{enumerate}
\item The pair $(\zz G, S)$ satisfies the left and right Ore condition, i.e. $S$ is a multiplicative closed subset of non-zero divisors with $1 \in S$ and fulfills
\[
S \cdot t \, \cap \, \zz G \cdot s \neq \emptyset \mbox{ and } t \cdot S \, \cap \, s \cdot \zz G \neq \emptyset
\]
for all $t \in \zz G$ and $s \in S$.
\item For any $s \in S$ the induced morphism
\[
r_s^{(2)}: l^{(2)}(G) \to l^{(2)}(G), u \mapsto us
\]
is a weak isomorphism with determinant ${\textstyle\det}_{{\mathcal N}G}(r_s^{(2)}) = 1$.
\end{enumerate}
\end{definition}

Notice that for a subset $S \subset \zz G$ with determinant localization property we can build the quotient ring $S^{-1} \cdot \zz G$. Any element in $S^{-1} \cdot \zz G$ can be written as $s_1^{-1}r_1$ and $r_2 s_2^{-1}$ with $s_1, s_2 \in S$, $r_1, r_2 \in \zz G$. Moreover, for $x_1, \ldots, x_n \in S^{-1} \cdot \zz G$ there exist elements $s_1, s_2 \in S$ satisfying $s_1 x_i \in \zz G$ and $x_i s_2 \in \zz G$ for all $1 \leq i \leq n$. For more details about quotient rings we refer to \cite[Chapter II, \S 1]{Ste75}.

\begin{definition}
Let $S \subset \zz G$ be a subset with determinant localization property. Let $A \in M(n \times n; S^{-1} \cdot \zz G)$. There exist elements $s_1, s_2 \in S$ with
\[
s_1 \cdot A \cdot s_2 \in M(n \times n; \zz G) \subset M(n \times n; S^{-1} \cdot \zz G).
\]
(Of course, it is possible to choose $s_1=1$ or $s_2=1$.) We call the matrix $A$ \emph{weakly invertible} if the induced map $r^{(2)}_{s_1 \cdot A \cdot s_2}: l^2(G)^n \to l^2(G)^n$ given by right multiplication with the matrix $s_1 \cdot A \cdot s_2$ is a weak isomorphism. In this case we define
\[
{\textstyle\det}_{S^{-1} \cdot \zz G}(A) := {\textstyle\det}_{{\mathcal N}G}(r^{(2)}_{s_1 \cdot A \cdot s_2}).
\]
\end{definition}

Let $s'_1, s'_2 \in S$ be further elements satisfying $s'_1 \cdot A \cdot s'_2 \in M(n \times n; \zz G)$. The Ore condition implies the existence of elements $s_3, s_4 \in S$ and $a_3, a_4 \in \zz G$ with $s_3 s_1 = a_3 s'_1$ and $s_2 s_4 = s'_2 a_4$. We conclude that $r^{(2)}_{a_3}, r^{(2)}_{a_4}: l^2(G)^n \to l^2(G)^n$ are weak isomorphisms with determinant 1. The equation
\begin{eqnarray*}
r^{(2)}_{a_4} \circ r^{(2)}_{s'_1 \cdot A \cdot s'_2} \circ r^{(2)}_{a_3} & = & r^{(2)}_{a_3 s'_1 \cdot A \cdot s'_2 a_4}\\
& = & r^{(2)}_{s_3 s_1 \cdot A \cdot s_2 s_4}\\
& = & r^{(2)}_{s_4} \circ r^{(2)}_{s_1 \cdot A \cdot s_2} \circ r^{(2)}_{s_3}
\end{eqnarray*}
implies that $r^{(2)}_{s'_1 \cdot A \cdot s'_2}: l^2(G)^n \to l^2(G)^n$ is a weak isomorphism with
\[
{\textstyle\det}_{{\mathcal N}G}(r^{(2)}_{s'_1 \cdot A \cdot s'_2}) = {\textstyle\det}_{{\mathcal N}G}(r^{(2)}_{s_1 \cdot A \cdot s_2}).
\]
This shows that $\det_{S^{-1} \cdot \zz G}(A)$ is well-defined.

The following lemma is a direct consequence of \cite[Lemma 9.7]{Luc02}.

\begin{lemma} \label{detlem}
Let $S \subset \zz G$ be a subset with determinant localization property.
\begin{enumerate}
\item Let $A, B \in M(n \times n; S^{-1} \cdot \zz G)$ be weakly invertible. Then
    \[
    {\textstyle\det}_{S^{-1} \cdot \zz G}(AB) = {\textstyle\det}_{S^{-1} \cdot \zz G}(A) \cdot {\textstyle\det}_{S^{-1} \cdot \zz G}(B). \label{detlem1}
    \]
\item Let $A \in M(n \times n; S^{-1} \cdot \zz G)$, $B \in M(n \times m; S^{-1} \cdot \zz G)$ and $D \in M(m \times m; S^{-1} \cdot \zz G)$. Suppose that $A$ and $D$ are weakly invertible. Then
    \[
    \left( \begin{array}{cc} A & B \\ 0 & D \end{array} \right) \in M((n+m) \times (n+m); S^{-1} \cdot \zz G)
    \]
    is weakly invertible with
    \[
    {\textstyle\det}_{S^{-1} \cdot \zz G}(\left( \begin{array}{cc} A & B \\ 0 & D \end{array} \right)) = {\textstyle\det}_{S^{-1} \cdot \zz G}(A) \cdot {\textstyle\det}_{S^{-1} \cdot \zz G}(D). \label{detlem2}
    \]
\end{enumerate}
\end{lemma}

Let $G$ be a group with semi-integral determinant (see Proposition \ref{classG}) and $S \subset \zz G$ a subset with determinant localization property. We obtain $\det_{S^{-1} \cdot \zz G}(A) \geq 1$ for any weakly invertible matrix $A \in M(n \times n; S^{-1} \cdot \zz G)$. Moreover, we conclude from Lemma \ref{detlem} (\ref{detlem1}) that $\det_{S^{-1} \cdot \zz G}(A) = 1$ for any invertible matrix $A$.

For the studying of $L^2$-torsion we introduce the notion of a weak chain contraction.

\begin{definition}
Let $C_*$ be a finite free based $S^{-1} \cdot \zz G$-chain complex. A \emph{weak chain contraction} for $C_*$ is a pair $(\gamma_*,u_*)$ such that
\begin{enumerate}
 \item $u_*: C_* \to C_*$ is a chain map,
 \item $u_p: C_p \to C_p$ is weakly invertible for all $p \in \zz$ and
 \item $\gamma_*: u_* \simeq 0$ is a chain homotopy satisfying $\gamma_* \circ u_* = u_* \circ \gamma_*$.
\end{enumerate}
\end{definition}

There is the following main example of a weak chain contraction.

\begin{example}\label{wcc-ex}
Let $X$ be a finite CW-complex such that $\pi_1(X)$ has semi-integral determinant and $b_n^{(2)}({\tilde X})=0$ for all $n$. Since 
\[
H_n^{(2)}({\tilde X}) = \ker(c^{(2)}_n({\tilde X})^* \circ c^{(2)}_n({\tilde X}) + c^{(2)}_{n+1}({\tilde X}) \circ c^{(2)}_{n+1}({\tilde X})^*),
\]
we conclude that
\[
\Delta_n := c^{(2)}_n({\tilde X})^* \circ c^{(2)}_n({\tilde X}) + c^{(2)}_{n+1}({\tilde X}) \circ c^{(2)}_{n+1}({\tilde X})^*: C^{(2)}_n({\tilde X}) \to C^{(2)}_n({\tilde X})
\]
is injective and hence a weak isomorphism. Hence, for any subset $S \subset \zz G$ with determinant localization property 
a weak chain contraction $(\gamma_*, u_*)$ for $S^{-1} \cdot C_*({\tilde X})$ is given by
\[
\gamma_n := c_{n+1}({\tilde X})^* \mbox{ and } u_n := c_n({\tilde X})^* \circ c_n({\tilde X}) + c_{n+1}({\tilde X}) \circ c_{n+1}({\tilde X})^*.
\]
\end{example}

\begin{lemma} \label{wcc-lem1}
Let $G$ be a group with semi-integral determinant and $S \subset \zz G$ a subset with determinant localization property. Let $(\gamma_*,u_*)$ be a weak chain contraction for a finite free based $S^{-1} \cdot \zz G$-chain complex $C_*$. We write $C_{\odd} := \oplus_{n \in \zz} C_{2n+1}$, $C_{\even} := \oplus_{n \in \zz} C_{2n}$.\\
Then the maps $(u c+\gamma)_{\odd}: C_{\odd} \to C_{\even}$, $(u c+\gamma)_{\even}: C_{\even} \to C_{odd}$, $u_{\odd}: C_{\odd} \to C_{\odd}$ and $u_{\even}: C_{\even} \to C_{\even}$ are weakly invertible and satisfy
\begin{eqnarray*}
{\textstyle\det}_{S^{-1} \cdot \zz G}(u_{\odd}) & = & {\textstyle\det}_{S^{-1} \cdot \zz G}(u_{\even}),\\
\ln \det((u  c+\gamma)_{\odd}) - \ln \det(u_{\odd}) & = & -\ln \det((uc+\gamma)_{\even}) + \ln \det(u_{\even}).
\end{eqnarray*}
Moreover, the number ${\textstyle\det}_{S^{-1} \cdot \zz G}((uc+\gamma)_{\odd}) - {\textstyle\det}_{S^{-1} \cdot \zz G}(u_{\odd})$ does not depend on the choice of the weak chain contraction $(\gamma_*,u_*)$.
\end{lemma}
\begin{proof}
Let $(\delta_*,v_*)$ be any weak chain contraction for $C_*$. We define $\theta: C_{\even} \to C_{\even}$ by
\[
\theta := v_* \circ u_* + \delta_* \circ \gamma_* = \left( \begin{array}{ccccc} \ddots & \vdots & \vdots & \vdots & \\ \cdots & v u & 0 & 0 & \cdots \\ \cdots & \delta \gamma & v u & 0 & \cdots \\ \cdots & 0 & \delta \gamma & v u & \cdots \\ & \vdots & \vdots & \vdots & \ddots \end{array} \right).
\]
A direct calculation shows that the composition
\[
\theta' := (v c+\delta)_{\even} \circ \theta \circ (u c+\gamma)_{\odd}: C_{\odd} \to C_{\odd}
\]
is given by the lower triangle matrix
\[
\left( \begin{array}{ccccc} \ddots & \vdots & \vdots & \vdots & \\ \cdots & (v^2 u^2)_{2n-1} & 0 & 0 & \cdots \\ \cdots & * & (v^2 u^2)_{2n+1} & 0 & \cdots \\ \cdots & * & * & (v^2 u^2)_{2n+3} & \cdots \\ & \vdots & \vdots & \vdots & \ddots \end{array} \right).
\]
Since $u_p$ and $v_p$ are weakly invertible for all $p \in \zz$, the maps $u_{\odd}$, $u_{\even}$, $v_{\odd}$, $v_{\even}$, $\theta$, $\theta'$, $(u c+\gamma)_{\odd}$ and $(v c+\delta)_{\even}$ are weakly invertible. From Lemma \ref{detlem} we conclude
\begin{eqnarray*}
&& 2 \cdot \ln \det(u_{\odd}) + 2 \cdot \ln \det(v_{\odd}) = \ln \det(\theta') =\\
&& \ln \det((u c+\gamma)_{\odd}) + \ln \det(\theta) + \ln \det((v c+\delta)_{\even}) =\\
&& \ln \det((u c+\gamma)_{\odd}) + \ln \det(u_{\even}) + \ln \det(v_{\even}) + \ln \det((v c+\delta)_{\even}).
\end{eqnarray*}
The equations
\[
(u c+\gamma)_{\odd} \circ u_{\odd} = u_{\even} \circ (uc+\gamma)_{\odd}, \quad (v c+\delta)_{\even} \circ v_{\even} = v_{\odd} \circ (v c+\delta)_{\even}
\]
imply
\[
{\textstyle\det}_{S^{-1} \cdot \zz G}(u_{\odd}) = {\textstyle\det}_{S^{-1} \cdot \zz G}(u_{\even}), \quad {\textstyle\det}_{S^{-1} \cdot \zz G}(v_{\even}) = {\textstyle\det}_{S^{-1} \cdot \zz G}(v_{\odd}).
\]
Hence we obtain the equation
\[
\ln \det((u c+\gamma)_{\odd}) - \ln \det(u_{\odd}) = - \ln \det((v c+\delta)_{\even}) + \ln \det(v_{\even}).
\]
This shows that the number ${\textstyle\det}_{S^{-1} \cdot \zz G}((uc+\gamma)_{\odd}) - {\textstyle\det}_{S^{-1} \cdot \zz G}(u_{\odd})$ does not depend on the choice of the weak chain contraction.
We obtain
\[
\ln \det((u c+\gamma)_{\odd}) - \ln \det(u_{\odd}) = -\ln \det((u c+\gamma)_{\even}) + \ln \det(u_{\even})
\]
by setting $(\delta_*,v_*) := (\gamma_*,u_*)$.
\end{proof}

Now we return to the situation described in Example \ref{wcc-ex}.

\begin{lemma} \label{wcc-lem2}
Let $X$ be a finite CW-complex such that $\pi_1(X)$ has semi-integral determinant and $b_n^{(2)}({\tilde X})=0$ for all $n$.
Let $S \subset \zz \pi_1(X)$ be a subset with determinant localization property and $(\gamma_*, u_*)$ a weak chain contraction for $S^{-1} \cdot C_*({\tilde X})$.
Then 
\[
\rho^{(2)}({\tilde X})= \ln {\textstyle\det}_{S^{-1} \cdot \zz G}((u c({\tilde X})+\gamma)_{\odd}) - \ln {\textstyle\det}_{S^{-1} \cdot \zz G}(u_{\odd}).
\]
\end{lemma}
\begin{proof}
Because of Lemma \ref{wcc-lem1} it suffices to prove the equation for the weak chain contraction $(\gamma_*, u_*)$ given by
\[
\gamma_n := c_{n+1}({\tilde X})^* \mbox{ and } u_n := c_n({\tilde X})^* \circ c_n({\tilde X}) + c_{n+1}({\tilde X}) \circ c_{n+1}({\tilde X})^*.
\]
Let $f_n: S^{-1} \cdot C_n({\tilde X}) \to S^{-1} \cdot C_n({\tilde X})$ by the $n$-fold composition $u_n \circ \ldots \circ u_n$. 
We write $C_{\odd} := \oplus_{n \in \zz} S^{-1} \cdot C({\tilde X})_{2n+1}$, $C_{\even} := \oplus_{n \in \zz} S^{-1} \cdot C({\tilde X})_{2n}$.
Notice that
\[
f_{\even} \circ (u c({\tilde X}) + c({\tilde X})^*)_{\odd} = (u c({\tilde X})^* + c({\tilde X}))_{\odd} \circ f_{\odd}.
\]
We conclude 
\begin{eqnarray*}
\rho^{(2)}({\tilde X}) & = & - \frac{1}{2} \cdot \sum_{n \geq 0} (-1)^n \cdot n \cdot \ln {\textstyle\det}_{S^{-1} \cdot \zz \pi_1(X)}(u_n)\\
& = & \frac{1}{2} \cdot \ln {\textstyle\det}_{S^{-1} \cdot \zz \pi_1(X)}(f_{\odd}) - \frac{1}{2} \cdot \ln {\textstyle\det}_{S^{-1} \cdot \zz \pi_1(X)}(f_{\even})\\
& = & \frac{1}{2} \cdot \ln \det((u c({\tilde X}) + c({\tilde X})^*)_{\odd}) - \frac{1}{2} \cdot \ln \det((u c({\tilde X})^* + c({\tilde X}))_{\odd}).
\end{eqnarray*}
Hence it remains to prove
\[
\ln \det((u c({\tilde X})^* + c({\tilde X}))_{\odd}) = -\ln \det((u c({\tilde X}) + c({\tilde X})^*)_{\odd}) + 2 \cdot \ln \det(u_{\odd}).
\]
Consider the dual chain complex $D_*$ given by $D_n := C_{-n}({\tilde X})$ and $d_n := c_{-n+1}({\tilde X})^*$.
It has the weak chain contraction $(c_{-*}({\tilde X}), u_{-*})$. From Lemma \ref{wcc-lem1} applied to $D_*$ we conclude that
\begin{eqnarray*}
\ln \det((u  c({\tilde X})^*+c({\tilde X}))_{\odd}) & = & -\ln \det((u c({\tilde X})^*+c({\tilde X}))_{\even}) + 2 \cdot \ln \det(u_{\odd})\\
& = & -\ln \det((u c({\tilde X})+c({\tilde X})^*)_{\odd}) + 2 \cdot \ln \det(u_{\odd})
\end{eqnarray*}
This finishes the proof of Lemma \ref{wcc-lem2}.
\end{proof}

\begin{theorem} \label{rhovan}
Let $X$ be a finite CW-complex. Let $S \subset \zz \pi_1(X)$ be a subset with determinant localization property satisfying
\[
S^{-1} \cdot H_n({\tilde X}) := S^{-1} \cdot \zz \pi_1(X) \otimes_{\zz \pi_1(X)} H_n({\tilde X}) = 0 \mbox{ for all } n \geq 0.
\]
Then $b_n^{(2)}({\tilde X})=0$ for all $n \geq 0$. Moreover, if $\pi_1(X)$ has semi-integral determinant then $\rho^{(2)}({\tilde X})=0$.
\end{theorem}
\begin{proof}
At first, we show $b_n^{(2)}({\tilde X})=0$ for all $n \geq 0$. Denote by ${\mathcal U}\pi_1(X)$ the algebra of operators affiliated to the group von Neumann algebra ${\mathcal N}\pi_1(X)$. It is the Ore localization of ${\mathcal N}\pi_1(X)$ with respect to all non-zero-divisors (see \cite[Proposition 2.8]{Rei99}). Hence there is an embedding $S^{-1} \cdot \zz \pi_1(X) \subset {\mathcal U}\pi_1(X)$. Since $S^{-1} \cdot \zz \pi_1(X)$ is a direct limit of free $\zz \pi_1(X)$-modules, $S^{-1} \cdot \zz \pi_1(X)$ is a flat right $\zz \pi_1(X)$-module. Hence
\[
H_n(S^{-1} \cdot C_*({\tilde X})) = S^{-1} \cdot H_n(C_*({\tilde X})) = S^{-1} \cdot H_n({\tilde X}) = 0.
\]
We conclude that $S^{-1} \cdot C_*({\tilde X})$ and hence ${\mathcal U}\pi_1(X) \otimes_{\zz \pi_1(X)} C_*({\tilde X})$ are contractible. Therefore, we obtain $H_n({\mathcal U}\pi_1(X) \otimes_{\zz \pi_1(X)} C_*({\tilde X})) = 0$. There exists a dimension function $\dim_{{\mathcal U}\pi_1(X)}$ for ${\mathcal U}\pi_1(X)$-modules satisfying
\[
b_n^{(2)}({\tilde X}) = \dim_{{\mathcal U}\pi_1(X)}(H_n({\mathcal U}\pi_1(X) \otimes_{\zz \pi_1(X)} C_*({\tilde X})))
\]
(see \cite[Proposition 4.2 (ii)]{Rei99}). This shows
\[
b_n^{(2)}({\tilde X}) = \dim_{{\mathcal U}\pi_1(X)}(H_n({\mathcal U}\pi_1(X) \otimes_{\zz \pi_1(X)} C_*({\tilde X}))) = 0.
\]
It remains to prove $\rho^{(2)}({\tilde X})=0$ under the assumption that $\pi_1(X)$ has semi-integral determinant. Since $S^{-1} \cdot C_*({\tilde X})$ is finite free acyclic, there exists a chain contraction $\delta_*$. Hence $(\delta_*,\id_*)$ is a weak chain contraction for $S^{-1} \cdot C_*({\tilde X})$.
From Lemma \ref{wcc-lem2} we conclude
\[
\rho^{(2)}({\tilde X})= \ln {\textstyle\det}_{S^{-1} \cdot \zz G}((c({\tilde X})+\delta)_{\odd}).
\]
Since $(c({\tilde X})+\delta)_{\even} \circ (c({\tilde X})+\delta)_{\odd}$ is given by a lower triangle matrix with the identity on the diagonal, we conclude that $(c({\tilde X})+\delta)_{\odd}$ is invertible. This shows $\det_{S^{-1} \cdot \zz G}((c({\tilde X})+\delta)_{\odd}) = 1$.
\end{proof}

Let $S_i \subset \zz G$ $(i \in I)$ be subsets with determinant localization property. Then $\langle S_i \mid i \in I \rangle \subset \zz G$, the multiplicative closed subset generated by all elements lying in any subset $S_i$, satisfies the left and right Ore condition by the following argument. Let $s \in \langle S_i \mid i \in I \rangle$ and $t \in \zz G$. The element $s$ can be written as $s = s_1 s_2 \cdots s_n$ with $s_k \in S_{i_k}$ for some $i_k \in I$. From the Ore condition for $(\zz G, S_{i_1})$ we conclude that there exist elements $s'_1 \in S_{i_1}$ and $t_1 \in \zz G$ with $s_1 t_1 = t s'_1$. Moreover, there exist elements $s'_k \in S_{i_k}$ and $t_k \in \zz G$ $(t = 2, \ldots, n)$ with $s_k t_k = t_{k-1} s'_k$ for all $2 \leq k \leq n$. We finally get $s t_n = t s'_1 s'_2 \cdots s'_n$. This shows $t \cdot \langle S_i \mid i \in I \rangle \, \cap \, s \cdot \zz G \neq \emptyset$. Analogously, one proves $\langle S_i \mid i \in I \rangle \cdot t \, \cap \, \zz G \cdot s \neq \emptyset$.
It follows that the subset $\langle S_i \mid i \in I \rangle \subset \zz G$ has again the determinant localization property.

In particular, we get the following result: For any group $G$ there exists a maximal subset $S_{max}(G) \subset \zz G$ with determinant localization property. It is the multiplicative closed subset generated by all elements lying in subsets with determinant localization property.

\begin{proposition} \label{propvan}
Let $H$ be a normal subgroup of $G$ and $M$ be a $\zz G$-module. If $S_{max}(H)^{-1} \cdot M = 0$ (where we consider $M$ as $\zz H$-module) then
\[
S_{max}(G)^{-1} \cdot M = 0.
\]
\end{proposition}
\begin{proof}
We will prove that $S_{max}(H) \subset \zz G$ has the determinant localization property. This implies $S_{max}(H) \subset S_{max}(G)$ and hence the statement follows.
Let $t \in \zz G$ and $s \in S_{max}(H)$. We can write $t = \sum_{i=1}^n t_i g_i$ with $t_i \in \zz H$ and $g_i \in G$. Since the subsets $g_i S_{max}(H) g_i^{-1} \in \zz H$ have the determinant localization property, we conclude $g_i S_{max}(H) g_i^{-1} \subset S_{max}(H)$. Therefore, $t_i (g_i s g_i^{-1})^{-1} \in S_{max}(H)^{-1} \cdot \zz H$ and we can choose $s' \in S_{max}(H)$ such that $s' t_i (g_i s g_i^{-1})^{-1} \in \zz H$ for all $i = 1, \ldots, n$. Define $f_i := s' t_i (g_i s g_i^{-1})^{-1} \in \zz H$. We finally get
\[
s' t = \sum_{i=1}^n s' t_i g_i = (\sum_{i=1}^n f_i g_i) s.
\]
This shows $S_{max}(H) \cdot t \, \cap \, \zz G \cdot s \neq \emptyset$. Analogously, we obtain $t \cdot S_{max}(H) \, \cap \, s \cdot \zz G \neq \emptyset$.
\end{proof}

\begin{theorem} \label{rhomain}
Let $X$ be a finite aspherical CW-complex such that the fundamental group $\pi_1(X)$ possesses a subnormal series
\[
\pi_1(X) \rhd G_m \rhd \ldots \rhd G_0
\]
with a non-trivial elementary amenable group $G_0$. Then
\[
b_n^{(2)}({\tilde X})=0 \mbox{ for all } n \geq 0.
\]
Moreover, if $\pi_1(X)$ has semi-integral determinant then $\rho^{(2)}({\tilde X})=0$.
\end{theorem}
\begin{proof}
Since $X$ is aspherical, we get $H_n({\tilde X})=0$ for $n \neq 0$ and $H_0({\tilde X}) = \zz$ with the trivial $\pi_1(X)$-operation. Because of Theorem \ref{rhovan} it suffices to show $S_{max}(\pi_1(X))^{-1} \cdot \zz = 0$.
We will first prove $S_{max}(\zz)^{-1} \cdot \zz = 0$. Let $z$ be a generator of the group $\zz$. Since $r_{z-1}^{(2)}: l^2(\zz) \to l^2(\zz)$ is a weak isomorphism with $\det_{{\mathcal N}\zz}(r_{z-1}^{(2)}) = 1$ (compare \cite[Example 1.8]{Luc02}), the subset
\[
\{ (z-1)^n \mid n \geq 0 \} \subset \zz [\zz]
\]
has the determinant localization property. Therefore, $z-1 \in S_{max}(\zz)$. Let $a \otimes b \in S_{max}(\zz)^{-1} \cdot \zz [\zz] \otimes_{\zz [\zz]} \zz$. We conclude
\[
a \otimes b = a (z-1)^{-1} \otimes (z-1) b = a (z-1)^{-1} \otimes 0 = 0.
\]
This shows $S_{max}(\zz)^{-1} \cdot \zz = 0$.
Since $G_0$ has a non-trivial torsionfree abelian normal subgroup (see Proposition \ref{HL}) and any non-trivial torsionfree abelian group has a normal subgroup isomorphic to $\zz$, Proposition \ref{propvan} implies $S_{max}(\pi_1(X))^{-1} \cdot \zz = 0$.
\end{proof}

There exists an alternative proof for the vanishing of the $L^2$-Betti numbers. From \cite[Theorem 10.12]{Luc02}) it follows that $b_n^{(2)}({\tilde X})=0$ for all $n \geq 0$, if $X$ is a finite aspherical CW-complex such that the fundamental group $\pi_1(X)$ possesses a subnormal series
\[
\pi_1(X) \rhd G_m \rhd \ldots \rhd G_0
\]
with a non-trivial amenable group $G_0$.

It would be interesting to know whether $S_{max}(G)^{-1} \cdot \zz = 0$ for all amenable groups $G$ of finite cohomological dimension. Notice that a positive answer would imply a generalization of Theorem \ref{rhomain} given by replacing ``elementary amenable'' with ``amenable''.

\end{document}